\documentclass[a4paper,12pt,doublespacing]{article}
\usepackage{epsfig,graphicx}
\usepackage{psfrag}
\usepackage{amssymb,bbm}
\usepackage{amsmath}
\usepackage{amsfonts}
\usepackage{graphicx}
\usepackage[latin5]{inputenc}
\usepackage{setspace}
\usepackage[toc,page]{appendix}
\usepackage{enumerate}
\usepackage[latin5]{inputenc}
\usepackage{blindtext}

\usepackage{fancyhdr} 

\fancyheadoffset{0cm}

\fancypagestyle{plain}{%
	\fancyhf{}%
	\fancyhead[R]{\thepage}%
}

\usepackage[justification=centering]{subfig}
\usepackage[colorlinks]{hyperref}
\newcounter{theorem}

\newtheorem{example}{Example}
\newtheorem{remark}{Remark}
\newtheorem{proposition}{Proposition}
\newtheorem{lemma}{Lemma}

\newenvironment{proof}[1][Proof]{\textbf{#1.} }{\rule{0.5em}{0.5em}}

\title{A Class of M\"{o}bius Iterated Function Systems \date{}}
\author{Gökçe ÇAKMAK, \footnote{Anadolu University, Science Faculty, Department of Mathematics, 26470,
Eskişehir, Turkey,  e-mails: gokcecakmak@anadolu.edu.tr, adeniz@anadolu.edu.tr, skocak@anadolu.edu.tr}
\thanks{Corresponding Author.}
 \and Ali DENİZ$^*$, \and  Şahin KOÇAK $^*$}

\begin{document}
\maketitle
\begin{abstract}
We give a procedure to produce Möbius iterated function systems  (MIFS) on the unit disc in the complex plane.
\end{abstract}
\section{Introduction}
\par Iterated function systems (IFS) devised by Hutchinson (\cite{Hutchinson}) are primary tools to generate (and understand) a large class of fractals. It is however not an easy matter to analyze systems going beyond the similarity transformations on Euclidean spaces. A particularly interesting class of IFS are Möbius iterated function systems (MIFS), but there are rather few examples of this class in the literature (see \cite{Vince}). In this note, we give a simple procedure to produce contractive M\"{o}bius functions on the unit disc in the complex plane which could serve for further investigations from various fractal points of view such as attractor properties, dimension considerations, tube formulas etc. Using Proposition \ref{prop2} below one can produce easily all Möbius transformations which map the unit disc of the complex plane contractively into itself and use them to construct MIFS's at will.
\par For the convenience of the reader, we first give the following lemma:
\begin{lemma} \label{lemma}
	Let $ \varphi(z)=\dfrac{az+b}{cz+d} $ $( ad-bc=1) $ be a M\"{o}bius transformation on $ \mathbb{C}\cup \{\infty\} $ and $ \mathcal{S}\subset \mathbb{C} $ be a circle with center $ M $ and radius $ R $. Then, $ \varphi(\mathcal{S}) $ is a circle with center $M'=\dfrac{(aM+b)(\overline{cM}+\overline{d})-R^2a\overline{c}}{(cM+d)(\overline{cM}+\overline{d})-R^2c\overline{c}}  $ and radius $ R'= \dfrac{R}{||cM+d|^2-R^2|c|^2|} $.
	\par (The circle $ \mathcal{S} $ is mapped onto a line iff $ cz+d $ vanishes at a point on $ \mathcal{S} $; this can also be expressed  as $ |cM+d|=R|c| $.)
\end{lemma}
\begin{proof}
	The circle $ \mathcal{S} $ can be expressed as $ |z-M|^2=R^2 $; if we denote \linebreak $ \varphi(z)=\dfrac{az+b}{cz+d} $  by $ w $ and insert $ z=\dfrac{dw-b}{a-cw} $ into the equation of the circle we obtain $  \left| \dfrac{dw-b}{a-cw}-M\right|^2=R^2 $, i.e.
	\[ (dw-b-M(a-cw))(\overline{dw}-\overline{b}-\overline{M}(\overline{a}-\overline{cw}))=R^2(a-cw)(\overline{a}-\overline{cw}) \]
	or
	\begin{eqnarray*}
		(|d|^2+\overline{c}d\overline{M}+c\overline{d}M+|cM|^2-R^2|c|^2)|w|^2
		-(\overline{b}d+\overline{a}d\overline{M}+\overline{b}cM+\overline{a}c|M|^2-R^2\overline{a}c)w
		\\-(b\overline{d}+b\overline{cM}+a\overline{d}M+a\overline{c}|M|^2-R^2a\overline{c})\overline{w}=R^2|a|^2
		-|b|^2-\overline{a}b\overline{M}-a\overline{b}M-|aM|^2.
	\end{eqnarray*}
This equation shows that $ \varphi(\mathcal{S}) $ is a circle with center \[ M'=\dfrac{b\overline{d}+b\overline{cM}+a\overline{d}M+a\overline{c}|M|^2-R^2a\overline{c} }{|d|^2+\overline{c}d\overline{M}+c\overline{d}M+|cM|^2-R^2|c|^2}. \]
\par With a straightforward calculation one can see that the radius of $ \varphi(\mathcal{S}) $ is\linebreak $ R'= \dfrac{R}{||cM+d|^2-R^2|c|^2|} $. (As the expression for $ R' $ shows, $ \mathcal{S} $ is mapped onto a line iff $ |cM+d|=R|c| $.)
\end{proof}

\section{Contractive Möbius Transformations on the Unit Disc}
We now give the following proposition which characterizes when a Möbius transformation $ \varphi(z)=\dfrac{az+b}{cz+d} $ maps the unit disc $ \mathbb{D} =\{ z\in \mathbb{C}:|z|\leq 1 \}$ contractively into itself. (``Contractively" means that there exists $ 0<s<1 $ such that $ |\varphi(z)-\varphi(w)|\leq s|z-w|$ for $z,w\in \mathbb{D}$.)

\begin{proposition} \label{prop1}
	Let $ \varphi(z)=\dfrac{az+b}{cz+d} $ $( ad-bc=1) $ be a M\"{o}bius transformation on $\mathbb{C}\cup \{\infty \}$.
	\begin{enumerate}
		\item $ \varphi $ maps the unit disc $ \mathbb{D} $ into itself $( \varphi(\mathbb{D})\subset \mathbb{D} )$ if and only if the following two conditions hold:
		\begin{enumerate} [i.]
			\item $ |c|<|d| $
			\item $ |a\overline{c}-b\overline{d}|+1\leq |d|^2-|c|^2 $
		\end{enumerate}
		\item $ \varphi $ maps $ \mathbb{D} $ contractively into itself if and only if the following two conditions hold:
		\begin{enumerate}[i.]
			\item $ |d|-|c|>1 $
			\item $ |a\overline{c}-b\overline{d}|+1\leq |d|^2-|c|^2 $
		\end{enumerate}
	\end{enumerate}
\end{proposition}

	\begin{proof}
		For the first part, we remark that $ \varphi $ maps the unit disc into itself if and only if $ |M'|+R' \leq 1 $, where $ M' $ is the center of $ \varphi(\mathbb{S})$ $ ( \mathbb{S}=\{ z\in \mathbb{C}: |z|=1 \}) $ and $ R' $ is the radius of $ \varphi(\mathbb{S})$; and additionally, $ \varphi(0) $ is mapped into the inside of $ \varphi(\mathbb{S})$, i.e. $ |\varphi(0)-M'|<R' $.
		\par By the lemma above (with $ M=0 $ and $ R=1 $) we have $ M'= \dfrac{b\overline{d}-a\overline{c}}{|d|^2-|c|^2}$ and $ R'=\dfrac{1}{||d|^2-|c|^2|} $. (For the circle $ \mathbb{S} $ to be mapped into a circle it must be $ |c|\neq |d| $.) Inserting the formulas for $ M' $ and $ R' $ into the above inequalities we get;
		\begin{enumerate}
			\item
			\begin{enumerate}[i.]
				\item
				$ 	|\varphi(0)-M'|=\left| \dfrac{b}{d}-\dfrac{b\overline{d}-a\overline{c}}{|d|^2-|c|^2} \right|=\dfrac{|bd\overline{d}-bc\overline{c}-bd\overline{d}+a\overline{c}d|}{|d|||d|^2-|c|^2|}=\dfrac{|c||ad-bc|}{|d|||d|^2-|c|^2|}\\=\dfrac{|c|}{|d|||d|^2-|c|^2|}<R'=\dfrac{1}{||d|^2-|c|^2|} , \text{which gives}   \;|c|<|d|. $
				\item
				$ |M'|+R' =\dfrac{|b\overline{d}-a\overline{c}|}{||d|^2-|c|^2|}+\dfrac{1}{||d|^2-|c|^2|} =\dfrac{|b\overline{d}-a\overline{c}|+1}{||d|^2-|c|^2|}\leq 1, \text{which gives} \\ |a\overline{c}-b\overline{d}|+1\leq |d|^2-|c|^2 \;\text{since} \; |c|<|d|. $
			\end{enumerate}
			\end{enumerate}
		Conversely, if these two conditions are satisfied then
		\begin{enumerate}
			\item
			\begin{enumerate}[i.]
				\item
				$ 	|\varphi(0)-M'|<R' \; \text{holds, since} \\ |\varphi(0)-M'|=\dfrac{|c|}{|d|||d|^2-|c|^2|}<\dfrac{1}{||d|^2-|c|^2|}=R' \; \text{by}\; |c|<|d|.$
				\item
				$ |M'|+R' \leq 1 \; \text{holds, since} \; |M'|+R'=\dfrac{|b\overline{d}-a\overline{c}|+1}{||d|^2-|c|^2|}=\dfrac{|b\overline{d}-a\overline{c}|+1}{|d|^2-|c|^2} \\ \text{by}\; |c|<|d|\; \text{and thus} \;|M'|+R'\leq 1\; \text{by the second condition.}\;  $
			\end{enumerate}
		\end{enumerate}
		\par We now consider the second part of the proposition. Let us first assume that the two given conditions are satisfied. Then by the first part of the proposition, the disc $ \mathbb{D} $ is mapped into itself. So we just need to show that $ \varphi $ is a contraction. Since $ \varphi'(z)=\dfrac{1}{(cz+d)^2} $ it will be enough to show that $ |cz+d|>1+\alpha $ for all $ z\in \mathbb{D}$ for some $ \alpha>0 $, because we could then write
		\begin{eqnarray*}
		&|\varphi(z)-\varphi(w)|=\left|\, \int \limits_{[z,w]}\varphi'(\zeta)d\zeta \right| \leq \int\limits_{[z,w]}|\varphi'(\zeta)|d\zeta  \leq \mathop{\max}_{\zeta\in [z,w]} |\varphi'(\zeta)||z-w|\\&<\dfrac{1}{(1+\alpha)^2}|z-w|\qquad  (\text{[z,w] is the segment connecting z and w.})
		\end{eqnarray*} 
		and thus $ \varphi$ would be a contraction on $ \mathbb{D} $. Indeed, $ |cz+d|\geq |d|-|cz|\geq |d|-|c|>1 $ and since $ \mathbb{D} $ is compact and $ |cz+d| $ is continuous, we can write $ |cz+d|>1+\alpha $ for some $ \alpha>0 $.
		\par  Conversely, let us assume that $ \varphi $ maps $ \mathbb{D} $ contractively into itself. By the first part of the proposition, the second condition holds and we must show the first condition $ |d|-|c|>1 $. Since $ \varphi $ is a contraction we can write  $ |\varphi(z)-\varphi(w)|\leq s|z-w| $ for some $ 0<s<1 $. This gives $ |\varphi'(z)|\leq s<1 $. Since $ |\varphi'(z)|=\dfrac{1}{|cz+d|^2} $ we get $ |cz+d|\geq \dfrac{1}{\sqrt s} >1 $. Inserting $ z=-\dfrac{|c|d}{c|d|} $ into this inequality, we get  $ \left| -c\dfrac{|c|d}{c|d|}+d \right|>1 $, $ \dfrac{|-d|c|+d|d||}{|d|}>1 $, thus $ |d|-|c|>1 $ because $ |d|>|c| $ by the first part of the proposition. (It is obvious that $ d $ cannot vanish because $ |d|>|c| $ and in the case of $ c=0 $ we get $ |d|>1 $ by $ |cz+d|>1$ and thus $ |d|-|c|>1.) $
	\end{proof}
	
\begin{proposition}\label{prop2}
	Let $ r\in \mathbb{R} $ with $ 0<r<1 $ and let $ m $ be a complex number satisfying $ |m|\leq 1-r $. Choose $ c\in \mathbb{C} $ with $ |c|<\dfrac{1-r}{2r} $ and $ d\in \mathbb{C} $ with $ |d|^2=|c|^2+1/r $. Let $ a=mc+r\overline{d} $ and $ b=md+r\overline{c} $.
	\par Then, $ \varphi(z)=\dfrac{az+b}{cz+d} $  maps $ \mathbb{D} $ contractively into itself. (Furthermore, the center of the image of unit circle is $ m $ and its radius is $ r $.)
\end{proposition}
	
	\begin{proof}
		By Proposition \ref{prop1} (part 2), it will be enough to show that $ |d|-|c|>1 $ and $ |a\overline{c}-b\overline{d}|+1\leq |d|^2-|c|^2  $. \par For the first condition, note that
		 \[
		|c|<\dfrac{1-r}{2r}=\frac{1}{2}\left( \dfrac{1}{r}-1\right) = \frac{1}{2}(|d|^2-|c|^2-1),
		\]
		which gives $ |c|^2+2|c|+1=(|c|+1)^2<|d|^2 $ and therefore $ |d|-|c|>1 $.
		\par For the second condition, inserting the expression for $ a $ and $ b $ we get
		\[
		a\overline{c}-b\overline{d} = (mc+r\bar{d})\overline{c} - (md+r\bar{c})\overline{d}=mc\bar{c}+r\bar{c}\bar{d}-md\bar{d}-r\bar{c}\bar{d}=m(|c|^2-|d|^2)
		\]
		and thus
		\begin{equation*}
			|a\overline{c}-b\overline{d}|+1=|m|\,||c|^2-|d|^2|+1=|m|(|d|^2-|c|^2)+1=\dfrac{|m|}{r}+1=\dfrac{|m|+r}{r} \leq \frac{1}{r}=|d|^2-|c|^2.
		\end{equation*}
		(To find the center of the image of the unit circle and its radius, insert $ a=mc+r\bar{d} $ and $ b=md+r\bar{c}$ into the formulas for $ M' $ and $ R' $ of Lemma \ref{lemma} with $ M=0 $ and $ R=1 $:	$ 	M'=\dfrac{b\bar{d}-a\bar{c}}{|d|^2-|c|^2}=\dfrac{(md+r\bar{c})\bar{d}-(mc+r\bar{d})\bar{c}}{|d|^2-|c|^2}=\dfrac{md\bar{d}-mc\bar{c}}{|d|^2-|c|^2}=m  $ and
	$ 	R'=\dfrac{1}{||d|^2-|c|^2|}=r.) $
	\end{proof}
\begin{remark}
Any Möbius transformation $ \varphi(z)=\dfrac{az+b}{cz+d} $ $( ad-bc=1) $ mapping the unit disc $ \mathbb{D} $ contractively into itself can be obtained in the way prescribed in Proposition \ref{prop2}. 
\end{remark}
\begin{proof}
	Since $ \varphi(\mathbb{D})\subset \mathbb{D} $, $ r $ and $ m $ are given as natural parameters; $r$ being the radius of $ \varphi(\mathbb{D})$ and $ m $ its center. They satisfy $ 0<r<1 $ ($ r>0 $ since $ \varphi $ is not constant and $ r<1 $ since $ \varphi $ is contractive) and $ |m|+r\leq 1 $. By Proposition \ref{prop1}, $ |d|-|c|>1 $ so that $ r $ can be expressed as $ r=\dfrac{1}{||d|^2-|c|^2|}=\dfrac{1}{|d|^2-|c|^2} $ by Lemma \ref{lemma} and we get $ |d|^2=|c|^2+1/r $. Inserting this into $ |d|-|c|>1 $ we get $ |c|<\dfrac{1-r}{2r} $. On the other hand, $ m=\dfrac{b\bar{d}-a\bar{c}}{|d|^2-|c|^2} $ by Lemma \ref{lemma}, hence $ m=(b\bar{d}-a\bar{c})\dfrac{1}{|d|^2-|c|^2}=(b\bar{d}-a\bar{c})r. $ Now, $ b\bar{d}-a\bar{c}=\dfrac{m}{r} $ and $ad-bc=1$ yield $ a=mc+r\overline{d} $ and $ b=md+r\overline{c} $.
\end{proof}
	\section{Examples}
	\par Using Proposition \ref{prop2} one can produce an infinitude of examples of Möbius iterated function systems. We give below several MIFS's by choosing $ r $, $ m $, $ c $ and $ d $ arbitrarily (only subject to the conditions of Proposition \ref{prop2}) and then determining $ a $ and $ b $ according to Proposition \ref{prop2}. We depict the corresponding attractors also (we used the software Cinderella \cite{Cinderella} to generate them). We remind that, for an IFS consisting of the contractions $ \varphi_1, \varphi_2,..., \varphi_N $ on $ \mathbb{D} $ (or on any complete metric space $X$), the Hausdorff limit $ A $ of the sequence $\varphi^k (Y)=\varphi\circ \varphi\circ...\circ\varphi (Y) $, where $ Y\subset X$ is any compact subset and $ \varphi(Y)=\varphi_1(Y)\cup \varphi_2(Y)\cup...\cup \varphi_N(Y) $, is called the attractor of the IFS $ \{ \varphi_1, \varphi_2,...,\varphi_N \}$. For generalities on iterated function systems see \cite{Barnsley}. 
	
	\begin{example}
		We determine $ \varphi_1$ by choosing $ r_1=0.7, m_1=-0.2+0.2i, c_1=0.2i,  d_1=1.0523+0.601i $; and $ \varphi_2$ by choosing $ r_2=0.6, m_2=-0.3i, c_2=-0.2i, d_2=-1.3064 $. Computing $ a_1, b_1 $ and $ a_2, b_2 $ according to the Proposition \ref{prop2}, one obtains \[ \varphi_1(z)=\dfrac{a_1z+b_1}{c_1z+d_1} =\dfrac{(0.6966-0.4607i)z+(-0.3307-0.0497i)}{(0.2i)z+(1.0523+0.601i)}\] and \[ \varphi_2(z)=\dfrac{a_2z+b_2}{c_2z+d_2}= \dfrac{(-0.8438)z+(0.5119i)}{(-0.2i)z+(-1.3064)} .\]  We show in Fig. \ref{fig1} (a) the images of the unit disc with respect to $ \varphi_1 $ and $ \varphi_2$, and in Fig. \ref{fig1} (b) the emerging attractor inside the unit disc.
	\end{example}

\begin{figure}[h!]
\centering
	\includegraphics[scale=0.4]{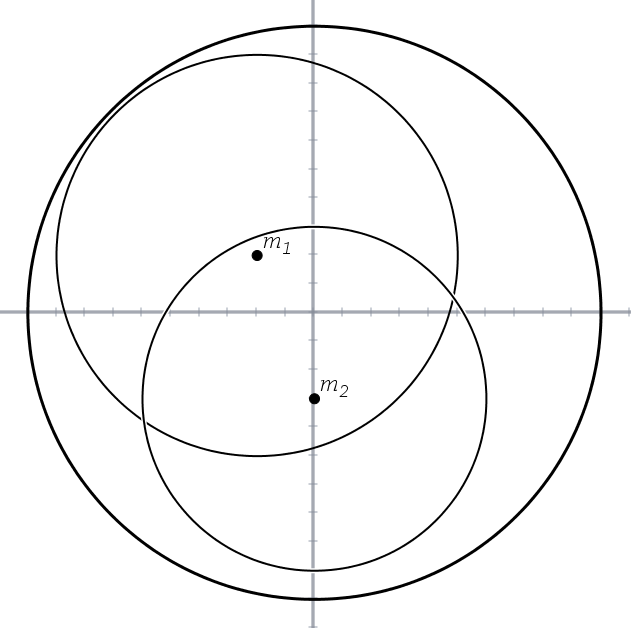}
	\includegraphics[scale=0.5]{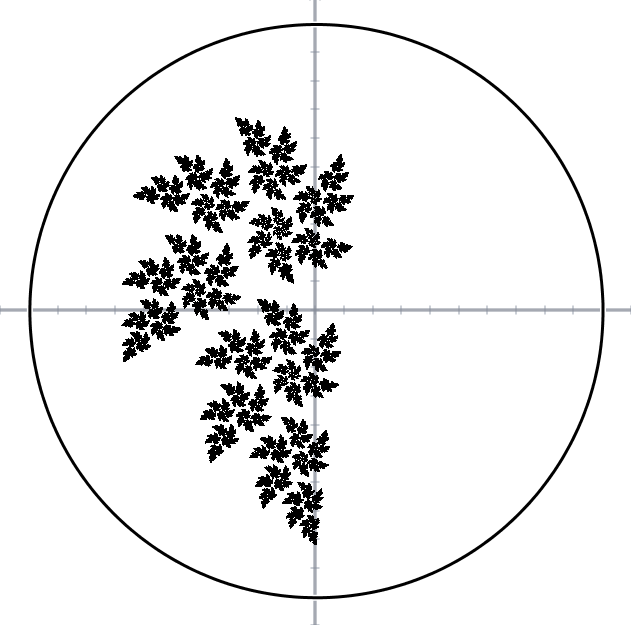}
\subfloat[]{\hspace{.35\linewidth}}
\subfloat[]{\hspace{.35\linewidth}}
\caption{The image of the unit circle with respect to $ \varphi_1 $ is shown with center $ m_1 $ and with respect to $ \varphi_2 $ with center $m_2 $ in (a). The emerging attractor is shown in (b). \label{fig1}}
\end{figure}

	\begin{example}
		We determine $ \varphi_1$ by choosing $ r_1=0.6, m_1=0.1+0.3i, c_1=0.2-0.2i , d_1=-0.0047+1.3216i $; $ \varphi_2$ by choosing $ r_2=0.4, m_2=0.3-0.4i, c_2=0.7i ,\linebreak d_2=0.9102+1.4702i $ and $ \varphi_3$ by choosing $ r_3=0.5, m_3=-0.4-0.2i, c_3=0.3+0.3i , d_3=-1.0108-1.0762i $. Computing $ a_1, b_1 $, $ a_2, b_2 $ and $ a_3, b_3 $ according to the Proposition \ref{prop2}, one obtains the transformations
		\[
		\varphi_1(z)=\dfrac{a_1z+b_1}{c_1z+d_1}=\dfrac{(0.0772-0.753i)z+(-0.277+0.2507i)}{(0.2-0.2i)z+(-0.0047+1.3216i)},
		\]
		\[
		\varphi_2(z)=\dfrac{a_2z+b_2}{c_2z+d_2}=\dfrac{(0.6441-0.3781i)z+(0.8611-0.203i)}{(0.7i)z+(0.9102+1.4702i)}
		\]
		and
		\[
		\varphi_3(z)=\dfrac{a_3z+b_3}{c_3z+d_3}=\dfrac{(-0.5654+0.3581i)z+(0.3391+0.4826i)}{(0.3+0.3i)z+(-1.0108-1.0762i)}.
		\]
		The emerging attractor inside the unit disc is shown in Fig. \ref{fig2}.
	\end{example}
\begin{figure}
	\centering
	\includegraphics[scale=0.4]{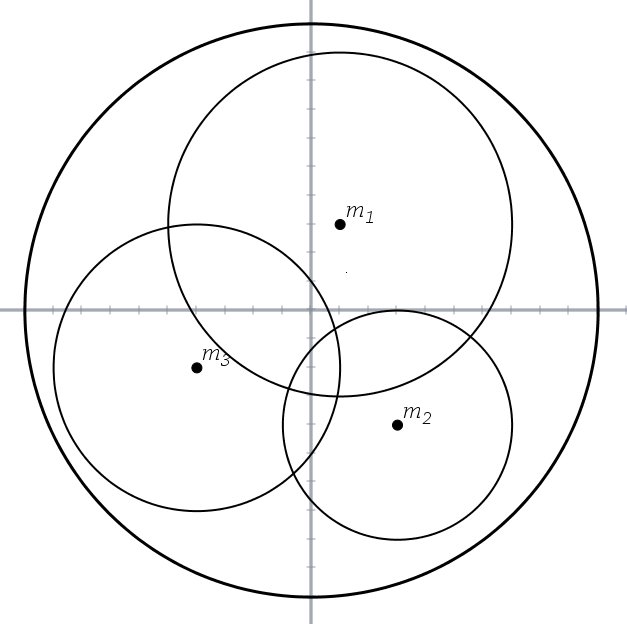}
	\includegraphics[scale=0.5]{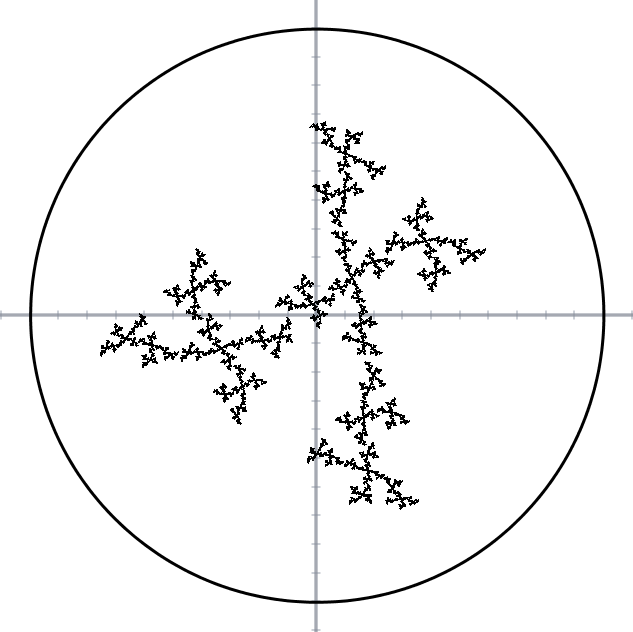}
	\subfloat[]{\hspace{.35\linewidth}}
	\subfloat[]{\hspace{.35\linewidth}}
	\caption{The image of the unit circle with respect to $ \varphi_1 $ is shown with center $ m_1 $, with respect to $ \varphi_2 $ with center $ m_2 $ and with respect to $ \varphi_3 $ with center $ m_3 $ in (a). The emerging attractor is shown in (b). \label{fig2}}
\end{figure}

	\begin{example}
	We determine $ \varphi_1$ by choosing $ r_1=0.6, m_1=-0.0116+0.3887i, c_1=-0.2033-0.1371i , d_1=-1.2886+0.2577i $; $ \varphi_2$ by choosing $ r_2=0.5, m_2=0.3991-0.2685i, c_2=0.2348-0.1316i , d_2=0.2431+1.4189i $ and $ \varphi_3$ by choosing $ r_3=0.5, m_3=-0.4114-0.2082i, c_3=-0.0445+0.3668i , d_3=1.2534-0.752i $. Computing $ a_1, b_1 $, $ a_2, b_2 $ and $ a_3, b_3 $ according to the Proposition \ref{prop2}, one obtains the transformations 
	\[
	\varphi_1(z)=\dfrac{a_1z+b_1}{c_1z+d_1}=\dfrac{(-0.7175-0.232i)z+(-0.2072-0.4216i)}{(-0.2033-0.1371i)z+(-1.2886+0.2577i)},
	\]
	\[
	\varphi_2(z)=\dfrac{a_2z+b_2}{c_2z+d_2}=\dfrac{(0.1799-0.825i)z+(0.5954+0.5668i)}{(0.2348-0.1316i)z+(0.2431+1.4189i)}
	\]
	and
	\[
	\varphi_3(z)=\dfrac{a_3z+b_3}{c_3z+d_3}=\dfrac{(0.7214+0.2344i)z+(-0.6945-0.135i)}{(-0.0445+0.3668i)z+(1.2534-0.752i)}.
	\]
	The emerging attractor inside the unit disc is shown in Fig. \ref{fig3}.
\end{example}
\begin{figure}
	\centering
	\includegraphics[scale=0.335]{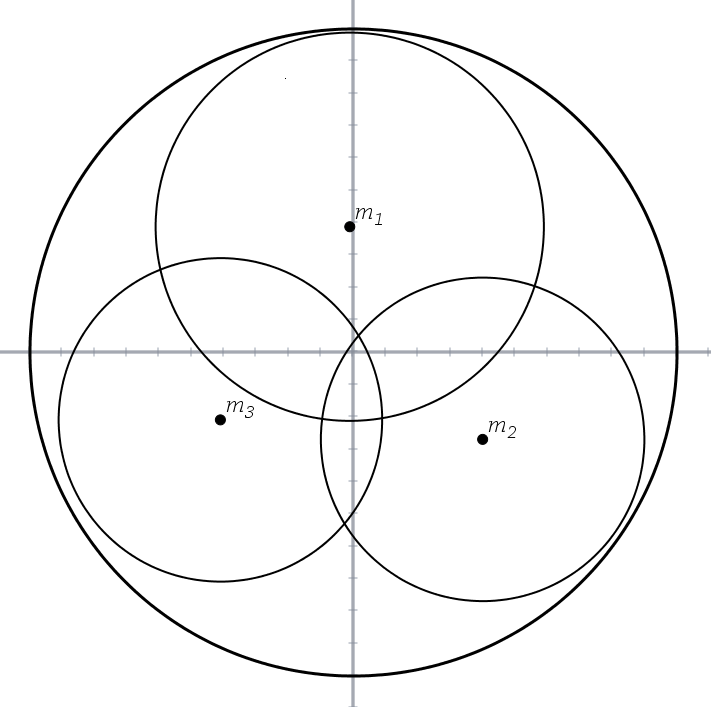}
	\includegraphics[scale=0.45]{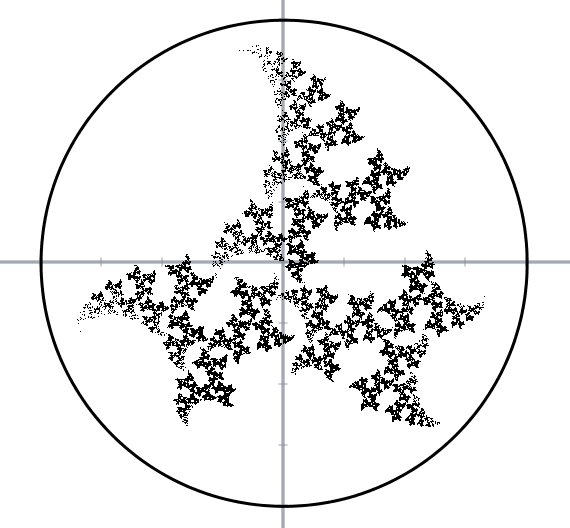}
	\subfloat[]{\hspace{.35\linewidth}}
	\subfloat[]{\hspace{.35\linewidth}}
	\caption{The image of the unit circle with respect to $ \varphi_1 $ is shown with center $ m_1 $, with respect to $ \varphi_2 $ with center $ m_2 $ and with respect to $ \varphi_3 $ with center $ m_3 $ in (a). The emerging attractor is shown in (b). \label{fig3}}

\end{figure}


\end{document}